\documentclass[12pt]{article}

\usepackage{amsmath}
\usepackage{amssymb}
\usepackage{amsthm}
\usepackage{setspace}
\usepackage{graphicx}
\usepackage[margin=2.54cm]{geometry}
\usepackage[pdftex,colorlinks,linkcolor=blue,citecolor=blue,urlcolor=blue]{hyperref}
\usepackage{mathtools}
\usepackage{siunitx}

\newcommand{\R}{\mathbb{R}}
\newcommand{\Sbb}{\mathbb{S}}

\newcommand{\F}{\mathcal{K}}
\newcommand{\G}{\mathcal{H}}

\newcommand{\Q}{\mathcal{Q}}
\newcommand{\Rc}{\mathcal{R}}

\newcommand{\SOC}{\mathcal{SOC}}
\newcommand{\RSOC}{\mathcal{RSOC}}
\newcommand{\PSD}{\mathcal{PSD}}
\renewcommand{\P}{\mathcal{P}}
\renewcommand{\S}{\mathcal{S}}

\newcommand{\Shor}{\text{\sc Shor}}
\newcommand{\RLT}{\text{\sc RLT}}
\newcommand{\SOCRLT}{\text{\sc SOCRLT}}
\newcommand{\Kron}{\text{\sc Kron}}
\newcommand{\Us}{\text{\sc Beta}}
\newcommand{\myJ}{J}

\DeclareMathOperator{\rank}{rank}

\DeclareMathOperator{\diag}{diag}

\DeclareMathOperator{\cone}{conic.hull}
\DeclareMathOperator{\Range}{Range}

\DeclareMathOperator{\bd}{bd}
\DeclareMathOperator{\myint}{int}

\DeclareMathOperator{\Arr}{Arr}

\newtheorem{lemma}{Lemma}
\newtheorem{theorem}{Theorem}
\newtheorem{example}{Example}

\title{A Slightly Lifted Convex Relaxation \\ for Nonconvex Quadratic Programming
\\ with Ball Constraints}

\author{%
Samuel Burer\thanks{Department of Business Analytics, University
of Iowa, Iowa City, IA, 52242--1994, USA\@. Email: {\tt
samuel-burer@uiowa.edu}.}%
}

\date{February 28, 2023 \\ Revised: \today}

\begin{document}

\maketitle

\begin{abstract}

\noindent Globally optimizing a nonconvex quadratic over the
intersection of $m$ balls in $\R^n$ is known to be polynomial-time
solvable for fixed $m$. Moreover, when $m=1$, the standard semidefinite
relaxation is exact. When $m=2$, it has been shown recently that an
exact relaxation can be constructed using a disjunctive semidefinite
formulation based essentially on two copies of the $m=1$ case. However,
there is no known explicit, tractable, exact convex representation
for $m \ge 3$. In this paper, we construct a new, polynomially sized
semidefinite relaxation for all $m$, {\color{black} which does not
employ a disjunctive approach. We show that our relaxation is exact
for $m=2$}. {\color{black} Then, for $m \ge 3$}, we demonstrate
empirically that it is {\color{black} fast and} strong compared to
existing relaxations. The key idea {\color{black} of the relaxation} is
a simple lifting of the original problem into dimension $n+1$. Extending
this construction: (i) we show that nonconvex quadratic programming
over $\|x\| \le \min \{ 1, g + h^T x \}$ has an exact semidefinite
representation; and {\color{black} (ii) we construct a new relaxation
for quadratic programming over the intersection of two ellipsoids,
which globally solves all instances of a benchmark collection from the
literature.}

\end{abstract}

\begin{onehalfspace}

\section{Introduction}

We study the nonconvex optimization problem
\begin{equation} \label{equ:origqp} \tag{QP}
    \min_{x \in \R^n}
    \left\{
        x^T Q x + 2 \, q^T x :
        \begin{array}{l}
            \|x - c_i \| \le \rho_i \quad \forall \ i=1,\ldots,m
        \end{array}
    \right\},
\end{equation}
where the data are the $n \times n$ symmetric matrix $Q$, column vectors
$q, c_1, \ldots, c_m \in \R^n$, and positive scalars $\rho_1, \ldots,
\rho_m \in \R$. In words, (\ref{equ:origqp}) is nonconvex quadratic
programming over the intersection of $m$ balls in $n$-dimensional space.
Without loss of generality, we assume $c_1 = 0$ and $\rho_1 = 1$, i.e.,
the first constraint is the unit ball. Note that, while the feasible set
of (\ref{equ:origqp}) is convex, the objective function is generally
nonconvex since $Q$ is not necessarily positive semidefinite. We assume that
(\ref{equ:origqp}) is feasible and hence has an optimal solution,
and we are specifically interested in strong convex relaxations of
(\ref{equ:origqp}).

Although (\ref{equ:origqp}) is polynomial-time solvable to within
$\epsilon$-accuracy for fixed $m$ \cite{Bienstock.2016}, there is no
known exact convex relaxation for all $m$. By {\em exact\/}, we mean
a relaxation with optimal value equal to that of (\ref{equ:origqp}).
When $m=1$, the standard semidefinite (SDP) relaxation is exact
\cite{Rendl.Wolkowicz.1997}. This relaxation is often called the
{\em Shor relaxation\/} and involves a single $(n + 1) \times
(n + 1)$ positive semidefinite variable. For $m=2$, Kelly et
al.~\cite{Kelly.et.al.2022} have recently shown that a particular
disjunctive semidefinite relaxation is exact. Their construction
is based on essentially two copies of the $m=1$ case, and as
a result, it utilizes two $(n + 1) \times (n + 1)$ positive
semidefinite variables. These and other SDP relaxations are detailed
in Section \ref{sec:background}. Our goals in this paper are: (i) to
construct a new, strong, non-disjunctive SDP relaxation, which is
polynomially sized in $n$ and $m$; (ii) {\color{black} to prove its
exactness for $m = 2$}; and (iii) to demonstrate empirically its effectiveness
{\color{black} for larger $m$}. In particular, by {\em
non-disjunctive\/}, we mean that the SDP optimizes over just one
semidefinite variable like the Shor relaxation---but with additional
constraints on that matrix variable.

Another line of research examines conditions on the data of
(\ref{equ:origqp}) under which the Shor relaxation is exact. Note
that, without loss of generality, by an orthogonal rotation in
the $x$ space, $Q$ may be assumed to be diagonal. In this case,
(\ref{equ:origqp}) is a special case of a {\em diagonal quadratically
constrained quadratic program (diagonal QCQP)\/}, that is, a QCQP in
which every quadratic function has a diagonal Hessian. Diagonal QCQPs
are well studied in the literature \cite{Burer.Ye.2020, Burer.Ye.2021,
Locatelli.2022, Wang.Kilinc-Karzan.2022, Azuma.et.al.2022},
where it is known, for example, that the Shor relaxation of
(\ref{equ:origqp}) is exact if $\text{sign}(q_j) = -\text{sign}(c_{1j})
= \cdots = -\text{sign}(c_{mj})$ for all $j = 1,\ldots,n$
\cite{Sojoudi.Lavaei.2014}. In this paper, however,
we seek a relaxation that is strong irrespective of the data.

(\ref{equ:origqp}) can also be optimized globally using any of the
high-quality global-optimization algorithms and software packages
available today. We are also aware of several papers studying global
approaches for (\ref{equ:origqp}) that take into account the problem's
specific structure \cite{Bienstock.Michalka.2014, Beck.Pan.2017,
Almaadeed.et.al.2022}. Each of these papers uses some combination of
enumeration and lower bounding to find a global optimal solution and
verify its optimality. In contrast, we are interested in computing a
single strong bound ({\color{black} exact for $m=2$, as mentioned above}),
and we will show that {\color{black} for $m \ge 3$} our relaxation is
frequently strong enough to deliver a global optimal solution via a
rank-1 SDP optimal solution. Moreover, our relaxation could certainly
be used as the lower-bounding technique within an enumerative scheme
to optimize (\ref{equ:origqp}) globally, but we leave this for future
research.

As we began this project, we were motivated by the idea of constructing
an exact relaxation for (\ref{equ:origqp}) {\color{black} for all $m$}.
Given that (\ref{equ:origqp}) is only known to be polynomial-time for
fixed $m$---not as a function of $m$---it was unclear whether a given
relaxation, which is polynomial in $m$, could possibly be exact for
all $m$. Ultimately, {\color{black} we will show that our relaxation is
exact for $m=2$ but} 
inexact for $m \ge 3$; see Section \ref{sec:balls}. Even still,
we believe that our new relaxation makes significant progress towards
approximating (\ref{equ:origqp}) {\color{black} for larger $m$} as we
demonstrate empirically.

Our approach is based on two simple transformations of the feasible set
of (\ref{equ:origqp}). First, we write the feasible set in the equivalent form
\[
    \left\{ x \in \R^n :
            x^T x \le \rho_i^2 - c_i^T c_i + 2 c_i^T x \quad \forall \ i=1,\ldots,m
    \right\}.
\]
Note that the $i$-th constraint {\color{black} is representable using a
rotated second-order-cone constraint of size $n+2$.} Second, we insert
an auxiliary variable $\beta \in \R$ within
each
constraint:
\[
    \left\{ {x \choose \beta} \in \R^{n+1} :
        \begin{array}{l}
            x^T x \le \beta \\
            \beta \le \rho_i^2 - c_i^T c_i + 2 c_i^T x \quad \forall \ i=1,\ldots,m
        \end{array}
    \right\}.
\]
Here, $\beta \le \rho_i^2 - c_i^T c_i + 2 c_i^T x$ is a linear constraint in $x$ and $\beta$. We can then equivalently
minimize $x^T Q x + 2 q^T x$ over feasible ${x \choose \beta} \in
\R^{n+1}$. {\color{black} The intuition for these transformations comes
from the fact that we are exchanging $m-1$ cone constraints in the
original problem with $m$ linear constraints.} In this sense, we have
greatly simplified the structure of the feasible set.

The paper is organized as follows. In Section \ref{sec:background}, we
introduce the required background necessary to build and evaluate SDP
relaxations of (\ref{equ:origqp}). We also discuss the literature on
relaxations specifically for (\ref{equ:origqp}), including the exact
relaxations mentioned above when $m = 1$ and $m = 2$. {\color{black}
Then, in Section \ref{sec:newrelax}, we introduce our new relaxation,
which we show to be exact for $m=2$ and empirically quite tight for
$m \ge 3$. (The proof for $m=2$ is relegated to Section \ref{sec:proofs}.) Next,
Section \ref{sec:extensions} considers two extensions.} First, we study
the related problem of nonconvex quadratic programming over the set
\[
    \left\{ x \in \R^n : 
    \begin{array}{l}
        \|x\| \le 1 \\
        \|x\| \le g + h^T x 
    \end{array}
    \right\},
\]
i.e., when linear functions bound the norm of $x$ instead of the
squared norm. This
is an interesting case in its own right, arising for
example as a substructure in the optimal power flow problem
\cite{Chen.et.al.2017,Eltved.Burer.2023} and also studied by Kelly et
al.~\cite{Kelly.et.al.2022}. Second, we extend our method to nonconvex
quadratic programming over the intersection of two general ellipsoids,
which is known as the {\em Celis-Dennis-Tapia problem\/} or the {\em two
trust-region subproblem (TTRS)\/}. See \cite{Consolini.Locatelli.2021}
and references therein for recent work on TTRS. We show empifically that
our approach solves all instances of a test set from the literature.
Finally, Section \ref{sec:proofs} contains the proof of our main
exactness result from Section \ref{sec:newrelax}.

We remark that all computational results in the paper were coded in
Python using the Fusion API of MOSEK 10.0.37 \cite{mosek} and run on a
M2 MacBook Air with 24 GB of RAM. Source code, scripts, and results
are available at \url{https://github.com/sburer/ballconstraints}.

\subsection{Notation and terminology}

Our notation and terminology is mostly standard. $\R^d$ is the space of
$d$-dimensional real column vectors, and $\Sbb^d$ is the space of $d
\times d$ real symmetric matrices. The identity matrix in $\Sbb^d$ is
denoted $I_d$, and the trace inner product on $\Sbb^d$ is defined as $M
\bullet N := \text{trace}(MN)$ for any two matrices $M,N \in \Sbb^d$.
Define
\begin{align*}
    \SOC^d &:= \left\{ v \in \R^d : \|(v_2, \ldots, v_d)\| \le v_1 \right\}, \\
    {\color{black} \RSOC^d} &:= {\color{black} \left\{ v \in \R^d : v_3^2 + \cdots +
    v_d^2 \le 2 v_1 v_2, \ v_1 \ge 0, \ v_2 \ge 0 \right\}}, \\
    \PSD^d &:= \left\{ M \in \Sbb^d : M \text{ is positive semidefinite} \right\}
\end{align*}
to be the $d$-dimensional second-order cone, {\color{black} the
$d$-dimensonal rotated second-order cone}, and the $d \times d$ positive
semidefinite cone, respectively. {\color{black} It is well known that
\[
    T_d \, \RSOC^d = \SOC^d, \quad \text{where} \quad
    T_d := \begin{pmatrix} 1/\sqrt2 & 1/\sqrt2 & 0 \\
    1/\sqrt2 & -1/\sqrt2 & 0 \\ 0 & 0 & I_{d-2} \end{pmatrix}
    \in \S^{d \times d},
\]
that is, $\RSOC^d$ is the image of $\SOC^d$ under the orthogonal
rotation given by $T_d$.} We use {\em SOC\/} and {\em PSD\/} as
abbreviations for {\em second-order cone\/} and {\em positive
semidefinite\/}. For any cone $S$, its {\em conic hull\/} is defined as
all finite sums of members in $S$, i.e., $\cone({\color{black} S}) := \{ \sum_{k=1}^K
s_k : K \in \mathbb{N}, s_k \in S \}$.

\section{Background on Semidefinite Relaxations} \label{sec:background}

In this section, we recount techniques and constructions from the
literature, which we will use for building convex relaxations in Section
\ref{sec:newrelax}. We also detail prior research on relaxations of
(\ref{equ:origqp}) specifically. Finally, we discuss standard ways
to measure the quality of a given relaxation on a particular problem
instance, which will be employed in Sections \ref{sec:newrelax}--\ref{sec:extensions}.

We caution the reader that we will reuse (or ``overload'') some notation
in this section. For example, the variable $x$ and dimension $n$ in
this section are not strictly the same as defined in the
Introduction.

\subsection{Techniques for building an SDP relaxation} \label{sec:build}

For this subsection as well as Section \ref{sec:meas}, we introduce the
generic nonconvex quadratic programming problem
\begin{equation} \label{equ:genericQP}
    \min \left\{ x^T Q x : x \in \F, \ x_1 = 1 \right\},
\end{equation}
where
\[
    \F := \left\{
        x \in \R^n :
        \begin{array}{l}
            \ell_1^T x \ge 0, \ \ell_2^T x \ge 0 \\
            L_3^T x \in \SOC^{d_3}, \ L_4^T x \in \SOC^{d_4} 
        \end{array}
    \right\}
\]
is a closed, convex cone with polyhedral and second-order-cone
constraints defined by the data vectors $\ell_1, \ell_2 \in \R^n$
and data matrices $L_3 \in \R^{n \times d_3}, L_4 \in \R^{n \times
d_4}$. Specifically, $\F$ is defined by two linear constraints and two
SOC constraints of different sizes. We assume that
{\color{black} $\{ x \in \F : x_1 = 1 \}$} is nonempty and bounded in which case (\ref{equ:genericQP}) has
an optimal solution. We also assume that the constraints
defining $\F$ imply $x_1 \ge 0$.

It is well known that (\ref{equ:genericQP}) is equivalent to
\begin{equation} \label{equ:genericQP2}
    \min \left\{ Q \bullet X : X \in \G, \ X_{11} = 1 \right\},
\end{equation}
where
\[
    \G := \cone \left\{ X = xx^T : x \in \F \right\} \subseteq \Sbb^n.
\]
{\em Equivalent\/} means that both have the same optimal
value and there exists a rank-1 optimal solution $X^* = x^*
(x^*)^T$ of (\ref{equ:genericQP2}), where $x^*$ is optimal for
(\ref{equ:genericQP}). Hence, a common approach to optimize or approximate
(\ref{equ:genericQP}) is to build strong convex relaxations of $\G$, and in
particular, semidefinite relaxations are a typical choice.

One standard option is the {\em Shor relaxation\/} of $\G$:
\[
    \text{\Shor} :=
    \left\{ X \in \PSD^n :
        \begin{array}{l}
        X e_1 \in \F \\
        J_{d_3} \bullet L_3^T X L_3 \ge 0 \\
        J_{d_4} \bullet L_4^T X L_4 \ge 0 
        \end{array}
    \right\},
\]
where $e_1 \in \R^n$ is the first unit vector and, for any dimension
$d$,
\[
    J_d :=
    \begin{pmatrix} 1 & 0 \\ 0 & -I_{d - 1} \end{pmatrix}
    \in \Sbb^d.
\]
Here, $X e_1 \in \F$ reflects the constraint $x \in \F$ of
(\ref{equ:genericQP}) and is derived from the implications
\[
    x \in \F
    \quad \Longrightarrow \quad
    x_1 \ge 0
    \quad \Longrightarrow \quad
    x_1 x \in \F 
    \quad \Longrightarrow \quad
    X e_1 \in \F
\]
where $e_1$ is the first standard unit vector.
Morevover, the linear constraint $J_{d_k} \bullet L_k^T X L_k \ge 0$ is
derived from the quadratic function defining $L_k^T x \in \SOC^{d_k}$.
The $\Shor$ relaxation with the constraint $X_{11} = 1$ possesses an important property. When the objective
matrix $Q$ is positive semidefinite on the subset of entries $x_2,
\ldots, x_n$ of $x$, then $\Shor$ is already strong enough to optimize the
quadratic problem, i.e., the optimal value of (\ref{equ:genericQP2})
equals the optimal value of (\ref{equ:genericQP}), and the embedded
solution $X e_1$ is optimal.

Next, we introduce an {\em RLT constraint\/} \cite{Sherali.Adams.1999}
using the implications
\[
    \ell_1^T x \ge 0, \ \ell_2^T x \ge 0
    \quad \Longrightarrow \quad
    \ell_1^T x x^T \ell_2 \ge 0
    \quad \Longrightarrow \quad
    \ell_1^T X \ell_2 \ge 0,
\]
and {\em SOCRLT constraints\/}
\cite{Sturm.Zhang.2003,Burer.Anstreicher.2013} using
\[
    \ell_i^T x \ge 0, \ L_k^T x \in \SOC^{d_k}
    \quad \Longrightarrow \quad
    L_k^T x x^T \ell_i \in \SOC^{d_k}
    \quad \Longrightarrow \quad
    L_k^T X \ell_i \in \SOC^{d_k}.
\]
Then defining
\begin{align*}
    \RLT &:= \left\{ X \in \Sbb^n : \ell_1^T X \ell_2 \ge 0 \right\}, \\
    \SOCRLT &:= \left\{ X \in \Sbb^n : L_k^T X \ell_i \in \SOC^{d_k}
    \ \ \forall \ i = 1,2, \ k = 3,4 \right\},
\end{align*}
we arrive at the strengthened relaxation
\[
    \Shor \cap \RLT \cap \SOCRLT.
\]

Another valid constraint can be derived from the fact that $L_k^T x \in
\SOC^{d_k}$ is equivalent to a positive semidefinite constraint (or {\em
linear matrix inequality\/}) \cite{Anstreicher.2017}. We first
write $y := L_3^T x$ and $z := L_4^T x$, so that the SOC constraints in
(\ref{equ:genericQP}) are $y \in \SOC^{d_3}$ and $z \in \SOC^{d_4}$.
Then it is well-known that
\[
    y \in \SOC^{d_3}
    \quad \Longleftrightarrow \quad
    \Arr_{d_3}(y) := 
    \begin{pmatrix}
        y_1 & y_2 & \cdots & y_{d_3} \\
        y_2 & y_1 & & \\
        \vdots & & \ddots & \\
        y_{d_3} & & & y_1
    \end{pmatrix} \in \PSD^{d_3}.
\]
For any dimension $d$, $\Arr_d : \R^d \to \Sbb^d$ is called the
{\em arrow operator\/}. Likewise $z \in \SOC^{d_4}$ if and only
if $\Arr_{d_4}(z) \in \PSD^{d_4}.$ Then, using the fact that the
Kronecker product of PSD matrices is PSD, we have $\Arr_{d_3}(y) \otimes
\Arr_{d_4}(z) \in \PSD^{d_3 d_4}$. Because the left-hand side of this
expression is quadratic in $y$ and $z$, we can equivalently write
\[
    \Arr_{d_3}(y) \otimes \Arr_{d_4}(z)
    =
    (\Arr_{d_3} \boxtimes \Arr_{d_4})(z y^T)
\]
where $\Arr_{d_3} \boxtimes \Arr_{d_4} : \R^{d_4 \times d_3} \to
\Sbb^{d_3 d_4}$ is an operator that is linear in $zy^T$ and defined by
\[
    (\Arr_{d_3} \boxtimes \Arr_{d_4})(z y^T)
    := 
    \begin{pmatrix}
        \Arr_{{\color{black} d_4}}(y_1 z) & \Arr_{{\color{black} d_4}}(y_2 z) & \cdots & \Arr_{{\color{black} d_4}}(y_{d_3} z) \\
        \Arr_{{\color{black} d_4}}(y_2 z) & \Arr_{{\color{black} d_4}}(y_1 z) & & \\
        \vdots & & \ddots & \\
        \Arr_{{\color{black} d_4}}(y_{d_3} z) & & & \Arr_{{\color{black} d_4}}(y_1 z)
    \end{pmatrix}.
\]
Substituting back $y = L_3^T x$ and $z = L_4^T x$, we arrive at the
implication
\[
    (\Arr_{d_3} \boxtimes \Arr_{d_4})(L_4^T xx^T L_3)
    \in \PSD^{d_3 d_4}
    \quad \Longrightarrow \quad
    (\Arr_{d_3} \boxtimes \Arr_{d_4})(L_4^T X L_3)
    \in \PSD^{d_3 d_4}.
\]
Then defining
\[
    \Kron := \left\{
        X \in \Sbb^n : 
        (\Arr_{d_3} \boxtimes \Arr_{d_4})(L_4^T X L_3) \in \PSD^{d_3 d_4}
    \right\},
\]
we have the strongest relaxation thus far:
\[
    \Shor \cap \RLT \cap \SOCRLT \cap \Kron.
\]

In Section \ref{sec:newrelax}, we will derive SDP relaxations for
(\ref{equ:origqp}) using the building blocks just introduced, i.e.,
$\Shor$ combined with $\RLT$, $\SOCRLT$, and $\Kron$. In a sense,
each of the three relaxations $\RLT$, $\SOCRLT$, and $\Kron$ have
the same goal---to combine information from pairs of constraints
in (\ref{equ:genericQP})---but differ due to the structure of the
underlying cones, i.e., the nonnegative orthant, the second-order cone,
and the PSD cone.

\subsection{Results from the literature} \label{sec:existing}

The relaxations \Shor, \RLT, \SOCRLT, and \Kron~just introduced are
known to be quite strong in a number of contexts closely related to
(\ref{equ:origqp}). Specifically, \Shor~is exact for the case of
(\ref{equ:origqp}) for $m=1$ \cite{Rendl.Wolkowicz.1997}, and when
a single linear constraint is added to the ball constraint $\|x\| \le
1$, then $\Shor \cap \SOCRLT$
is exact \cite{Sturm.Zhang.2003}. In addition, for the more general
case
\[
    \left\{ x :
        \begin{array}{l}
        \|x\| \le 1 \\
         0 \le g_i + h_i^T x \ \ \forall \ i=2,\ldots,m
        \end{array}
    \right\}
\]
in which multiple linear constraints are added, $\Shor \cap \RLT
\cap \SOCRLT$ is {\color{black} guaranteed to be exact when} none
of the hyperplanes $0 = g_i + h_i^T x$ intersect inside the ball
\cite{Burer.Yang.2015}. \Kron~has also been used to strengthen
relaxations of (\ref{equ:origqp}) when there is a second ellipsoidal
constraint \cite{Anstreicher.2017}, not necessarily a ball; this is TTRS
mentioned in the Introduction. As a footnote, we are unaware of any case
in which enforcing \Kron~makes a relaxation exact for all
objectives, but nevertheless \Kron~has proven to be a valuable tool for
building strong relaxations.

We are aware of two papers \cite{Jiang.Li.2019,Eltved.Burer.2023}, which
have studied techniques for further strengthening $\Shor \cap \RLT \cap
\SOCRLT \cap \Kron$. Among these, \cite{Eltved.Burer.2023} is more
similar to the current paper. In \cite{Eltved.Burer.2023}, the authors
introduce a class of linear cuts for SDP relaxations of
\[
    \{ x : \|x\| \le 1, \ \|x - c_2\| \le g_2 + h_2^T x \},
\]
and they show how to separate the cuts in polynomial time. Through a
series of computational experiments, the authors also show that their
cuts are effective particularly in lower dimensions, say, for $n \le
10$.

As mentioned in the Introduction, an exact convex relaxation of
(\ref{equ:origqp}) via a disjunctive {\color{black} formulation for the
case $m=2$} was recently given by Kelly et al.~\cite{Kelly.et.al.2022}.
They showed that the feasible region of (\ref{equ:origqp}) can be
written as the union of two special sets:
\[
\{ x : \|x\| \le 1, \ 0 \le g_1 + h_1^T x \}
\quad\quad \text{and} \quad\quad
\{ x : \|x - c_2\| \le \rho_2, \ 0 \le g_2 + h_2^T x \}.
\]
Explicit formulas for $g_i$ and $h_i$ are given in their paper. Since
optimizing over each of these can separately be accomplished with $\Shor
\cap \SOCRLT$ as mentioned above, the authors then use a disjunctive
formulation with two copies of $\Shor \cap \SOCRLT$ to derive an exact
formulation of (\ref{equ:origqp}); see proposition 5 in their paper.
The authors also used similar ideas to derive an exact, disjunctive
formulation for the case $\|x\| \le \min \{1, g + h^T x \}$,
{\color{black} a case we will also consider in Section \ref{sec:extensions}}.

Zhen et al.~\cite{Zhen.et.al.2021} have recently suggested another
technique for combining information from two SOC constraints; see
appendix B of \cite{Zhen.et.al.2021}. We illustrate their idea using
the two constraints $\|x\| \le 1$ and $\|x - c_2\| \le \rho_2$ of
(\ref{equ:origqp}). The fact that $\|uv^T\|_2 = \|u\| \|v\|$,
where $\|\cdot\|_2$ denotes the matrix 2-norm, implies
\[
    \|xx^T - xc^T \|_2 = \|x(x - c)^T\|_2
    = \|x\| \|x - c_2\| \le 1 \cdot \rho_2 = \rho_2,
\]
which can be linearized
\[
    \|X - xc^T\|_2 \le \rho_2
    \quad\quad \Longleftrightarrow \quad\quad
    \begin{pmatrix}
        \rho_2 I_n & X - xc^T \\
        X - cx^T & I_n
    \end{pmatrix} \in \PSD^{2n}.
\]
In our experiments in Sections \ref{sec:newrelax}, we added this
constraint to $\Shor \cap \RLT \cap \SOCRLT \cap \Kron$, but it did
not provide added strength on our test instances. Although we do not
consider this valid constraint further in this paper, investigating its
precise relationship with existing constraints remains an interesting
avenue for research.

\subsection{Measuring the quality of a relaxation} \label{sec:meas}

Following the notation of Section \ref{sec:build}, let $\Rc$ be a given
semidefinite relaxation of the conic hull $\G$, which is at least as
strong as $\Shor$, i.e., $\G \subseteq \Rc \subseteq \Shor$. The SDP
relaxation corresponding to $\Rc$ is
\[
    r^* := \min \left\{ Q \bullet X : X \in \Rc, \ X_{11} = 1 \right\},
\]
and we let $X^*$ denote an optimal solution.

We can assess the quality of the relaxation by comparing $r^*$ to 
any readily available feasible value $v = x^T Q x$, where $x
\in \F$ with $x_1 = 1$ is some feasible point. In particular, as mentioned in Section
\ref{sec:build}, $X^* e_1$ is an embedded feasible solution, but there
are often multiple methods for obtaining a good feasible value $v$ in
practice, e.g., using a rounding procedure from $X^*$ or some other type of
heuristic. Our primary measure of relaxation quality will
be the {\em relative gap\/} between $r^*$ and $v$ defined as follows:
\begin{equation} \label{equ:relgap}
    \text{relative gap} := \frac{ v - r^* }{ \max \{ 1, \tfrac12 |v + r^*| \} }.
\end{equation}

A secondary measure of relaxation quality is the so-called {\em
eigenvalue ratio\/} of $X^*$. By construction, if the rank of $X^*$
is 1, then $X^* = x^*(x^*)^T$ for some optimal solution $x^*$ of
(\ref{equ:genericQP}). Of course, in practice $X^*$ will most likely not
be exactly rank-1, but it may be close numerically. The eigenvalue ratio
tries to capture exactly how close. Specifically,
\[
    \text{eigenvalue ratio} := \frac{ \lambda_1[X^*] }{ \lambda_2 [X^*] },
\]
where $\lambda_1[X^*]$ and $\lambda_2 [X^*]$ are the largest and
second-largest eigenvalues of $X^*$. Generally speaking, the higher
the eigenvalue ratio, the closer $X^*$ is to being rank-1.

In the computational results of Sections
\ref{sec:newrelax}--\ref{sec:extensions}, we will say that an instance
is {\em solved\/} by a relaxation if both of the following two
conditions are satisfied:

\begin{itemize}

\item the relative gap between $r^*$ and the feasible value $v$, which
comes from the solution $x := X^* e_1$ embedded in $X^*$, is less than
$10^{-4}$;

\item the eigenvalue ratio of $X^*$ is greater than $10^4$.

\end{itemize}

\noindent Similar definitions of the term {\em solved\/} have been used in
\cite{Burer.Anstreicher.2013,Almaadeed.et.al.2022,Eltved.Burer.2023,Consolini.Locatelli.2021}.
Strictly speaking, a small relative gap is enough to verify approximate
optimality, but we also require a large eigenvalue ratio in order to
bolster our confidence in the numerical results. It should also be noted
that, for randomly generated instances such as those investigated in
Section \ref{sec:newrelax}, small relative gaps and large eigenvalue
ratios are typically highly positively correlated.

\section{A New Relaxation} \label{sec:newrelax}

In this section, we tailor the ideas of Section \ref{sec:background} to
derive a new relaxation of (\ref{equ:origqp}), prove it is exact for
$m=2$, and investigate its empirical performance for larger $m$.

\subsection{An existing relaxation}

As
discussed in Section \ref{sec:background}, we can combine existing
techniques to build a first relaxation of (\ref{equ:origqp}). Consider the feasible set of
(\ref{equ:origqp}):
\begin{equation} \label{equ:groundset}
    \left\{ x \in \R^n :
    \| x - c_i \| \le \rho_i \ \ \forall \ i = 1,\ldots,m
    \right\}.
\end{equation}
We first homogenize (\ref{equ:groundset}) by introducing $\alpha \in
\R$:
\begin{align*} 
    \widetilde \F
    &:= \left\{ {\alpha \choose x} \in \R^{n+1} : 
    \begin{array}{l}
    \|x - \alpha c_i \| \le \rho_i \alpha \ \ \forall \ i = 1,\ldots,m
    \end{array}
    \right\} \\
    &:= \left\{ \widetilde w \in \R^{n+1} : 
        \widetilde L_i^T w \in \SOC^{n+1}
        \quad \forall \ i =1,\ldots,m
    \right\},
\end{align*}
where $\widetilde w := {\alpha \choose x} \in \R^{n+1}$
and
\[
    \widetilde{L}_i := \begin{pmatrix} g_i & -c_i^T \\ h_i & I_n \end{pmatrix}
    \in \R^{(n + 1) \times (n + 1)} \ \ \forall \ i=1,\ldots,m.
\]
Note that $\alpha \ge 0$ is ensured because each $\rho_i > 0$.
Then, in accordance with Section \ref{sec:background}, we define
the following relaxation of $\widetilde\G := \cone \{ \widetilde w
\widetilde w^T \in \Sbb^{n+1} : \widetilde w \in \widetilde \F \}$:
\begin{align*}
    \widetilde \Rc &:= \Shor \cap \Kron \\
                   &:= \left\{
        \widetilde W \in \PSD^{n+1} :
        \begin{array}{ll}
            J_{n+1} \bullet \widetilde{L}_i^T \widetilde{W} \widetilde{L}_i \ge 0 & \forall \ i=1,\ldots,m \\
            (\Arr_{n+1} \boxtimes \Arr_{n+1})(\widetilde L_k^T \widetilde W
            \widetilde L_i) \in \PSD^{(n + 1)^2} & \forall \ 1 \le i < k \le m
        \end{array}
    \right\}.
\end{align*}
Here, $\RLT$ and $\SOCRLT$ are not applicable because $\widetilde
\F$ contains no explicit linear constraints.

To approximate the problem (\ref{equ:origqp}), we then have the
semidefinite relaxation
\begin{equation}
    \min \left\{ \widetilde Q \bullet \widetilde W : \widetilde W \in
    \widetilde \Rc, \ \widetilde W_{11} = 1 \right\} \tag{\Kron}
\end{equation}
where
\[
    \widetilde Q := \begin{pmatrix} 0 & q^T \\ q & Q \end{pmatrix}.
\]
We use the name ``\Kron'' to remind the reader that $\widetilde \Rc$ is
equivalent to $\Shor \cap \Kron$.

\subsection{Our new relaxation}

As discussed in the Introduction, in hopes of improving $\widetilde
\Rc$, we introduce the auxiliary variable $\beta \in \R$ into
(\ref{equ:groundset}):
\[
    \left\{ {x \choose \beta} \in \R^{n+1} : 
    \begin{array}{l}
    x^T x \le \beta \\
    \beta \le  \rho_i^2 - c_i^T c_i + 2 c_i^T x \ \ \forall \ i = 1, \ldots, m
    \end{array}
    \right\},
\]
where $x^T x \le \beta$ is SOC-representable. Compared to
(\ref{equ:groundset}), this swaps $m-1$ conic constraints for $m$ linear
constraints, hence simplifying the structure of the feasible set.
However, it does not affect optimization over (\ref{equ:groundset})
since $\beta$ is an artificial variable not involved in the objective
$x^T Q x + 2 q^T x$.

Similar to above, we proceed by adding the redundant $\beta \ge 0$ and
homogenizing with $\alpha \ge 0$:
\[
    \F := \left\{
        w \in \R^{n+2} : 
    \begin{array}{l}
        x^Tx \le \alpha \beta, \quad \alpha \ge 0, \quad \beta \ge 0 \\
        \beta \le (\rho_i^2 - c_i^T c_i) \alpha + 2 c_i^T x \ \ \forall \ i = 1,\ldots,m
    \end{array}
    \right\},
\]
where
\[
    w := \begin{pmatrix} \alpha \\ x \\ \beta \end{pmatrix} \in \R^{n+2}.
\]
We next define
\[
    P :=
    \begin{pmatrix}
        {\color{black} 1/2} & 0 & 0 \\
        0 & 0 & I_n \\
        0 & 1 & 0
    \end{pmatrix}
    {\color{black} T_{n+2}}
    \in \R^{(n + 2) \times (n + 2)}
\]
and
\[
    \ell_i := 
    \begin{pmatrix} \rho_i^2 - c_i^T c_i \\ 2 c_i \\ -1 \end{pmatrix} \in \R^{n+2}
    \ \ \forall \ i = 1,\ldots,m
\]
so that $\F$ can be expressed more compactly as
\[
    \F 
    = \left\{
    w \in \R^{n+2} :
    P^T w \in \SOC^{n+2}, \
    \ell_i^T w \ge 0 \ \forall \ i = 1,\ldots,m
    \right\}.
\]
In particular, we have used the equivalence
\[
    \begin{array}{c}
        x^T x \le \alpha \beta \\
        \alpha \ge 0, \ \beta \ge 0
    \end{array}
    \quad \Longleftrightarrow \quad
    \begin{pmatrix} \alpha/2 \\ \beta \\ x \end{pmatrix} \in \RSOC^{n+2}
    \quad \Longleftrightarrow \quad
    P^T w \in \SOC^{n+2}.
\]

Then Section \ref{sec:background} provides the following relaxation of
$\G := \cone \{ ww^T \in \Sbb^{n+2} : w \in \F \}$:
\begin{align*}
    \Rc
    &:= \Shor \cap \RLT \cap \SOCRLT \\
    &:=
    \left\{
        W \in \PSD^{n+2}:
        \begin{array}{l}
            J_{n+2} \bullet P^T W P \ge 0 \\
            \ell_i^T W \ell_k \ge 0 \quad\quad\quad\quad
            \forall \ 1 \le i < k \le m \\
            P^T W\ell_i \in \SOC^{n+2} \, \ \ \forall \ i=1,\ldots,m
        \end{array}
    \right\}.
\end{align*}
Here, the Kronecker approach is not applicable because $\F$ only contains one SOC
constraint. Furthermore, compared to $\widetilde \Rc$, the relaxation
$\Rc$ has only one PSD constraint but contains roughly ${m \choose 2}$
linear constraints and $m$ SOC constraints.

To approximate the problem (\ref{equ:origqp}), we then have the
semidefinite relaxation
\begin{equation}
    \min \left\{ \widehat Q \bullet W : W \in \Rc, \ W_{11} = 1 \right\}, \tag{\Us}
\end{equation}
where
\[
    \widehat Q := \begin{pmatrix} 0 & q^T & 0 \\ q & Q & 0 \\ 0 & 0 & 0 \end{pmatrix}.
\]
We use the name ``\Us'' to denote our relaxation, which is based on the
artificial variable $\beta$. The following small example shows that
\Us~can significantly improve \Kron.

\begin{example} \label{exa:1}
Consider an instance of (\ref{equ:origqp}) with $(n,m) = (2,2)$ and
\[
    Q = \begin{pmatrix} -0.6 & 0 \\ 0 & -0.44 \end{pmatrix}, \quad
    q = \begin{pmatrix} -0.03 \\ 0 \end{pmatrix}, \quad
    c_2 = \begin{pmatrix} -0.3 \\ -0.3 \end{pmatrix}, \quad
    \rho_2 = 1.
\]
\Shor~returns the lower bound $-0.5876$, \Kron~returns $-0.5487$, and \Us~solves the
instance with an optimal value of $-0.54$ and optimal solution
\[
    x^* = \begin{pmatrix}
    -1 \\
     0
    \end{pmatrix}.
\]

In Figure \ref{fig:example_1}, we plot four visualizations of this
example. Each subplot depicts the minimum value of the specified
optimization problem when the variable $x_2$ is fixed to a specific
value in the interval $[-1.0, 0.6]$.

For example, in the top-left subplot, we plot the minimum value of the
original quadratic objective $x^T Q x + 2 q^T x$ over the original
feasible set intersected with the constraint fixing $x_2$ fixed at a value along the horizontal axis.
This creates a one-dimensional graph as $x_2$ varies from $-1.0$ to
$0.6$. In the remaining subplots, the same procedure is applied to
the $\Shor$, $\Kron$, and $\Us$ relaxations and the first subplot is
repeated as a dashed curve for reference.

Overall, these plots depict the lower envelope of the associated
feasible sets projected down to the 2-dimensional space of $x_2$ and
the objective value. As such, they illustrate the nonconvexity of the
original objective as well as the increasing strength of the three
convex relaxations. In particular, $\Shor$ provides a relatively weak
convex underestimator, $\Kron$ provides a tighter underestimator, and
$\Us$ provides the convex envelope.

\end{example}

\begin{figure}
    \centering
    \includegraphics[width=5.5in]{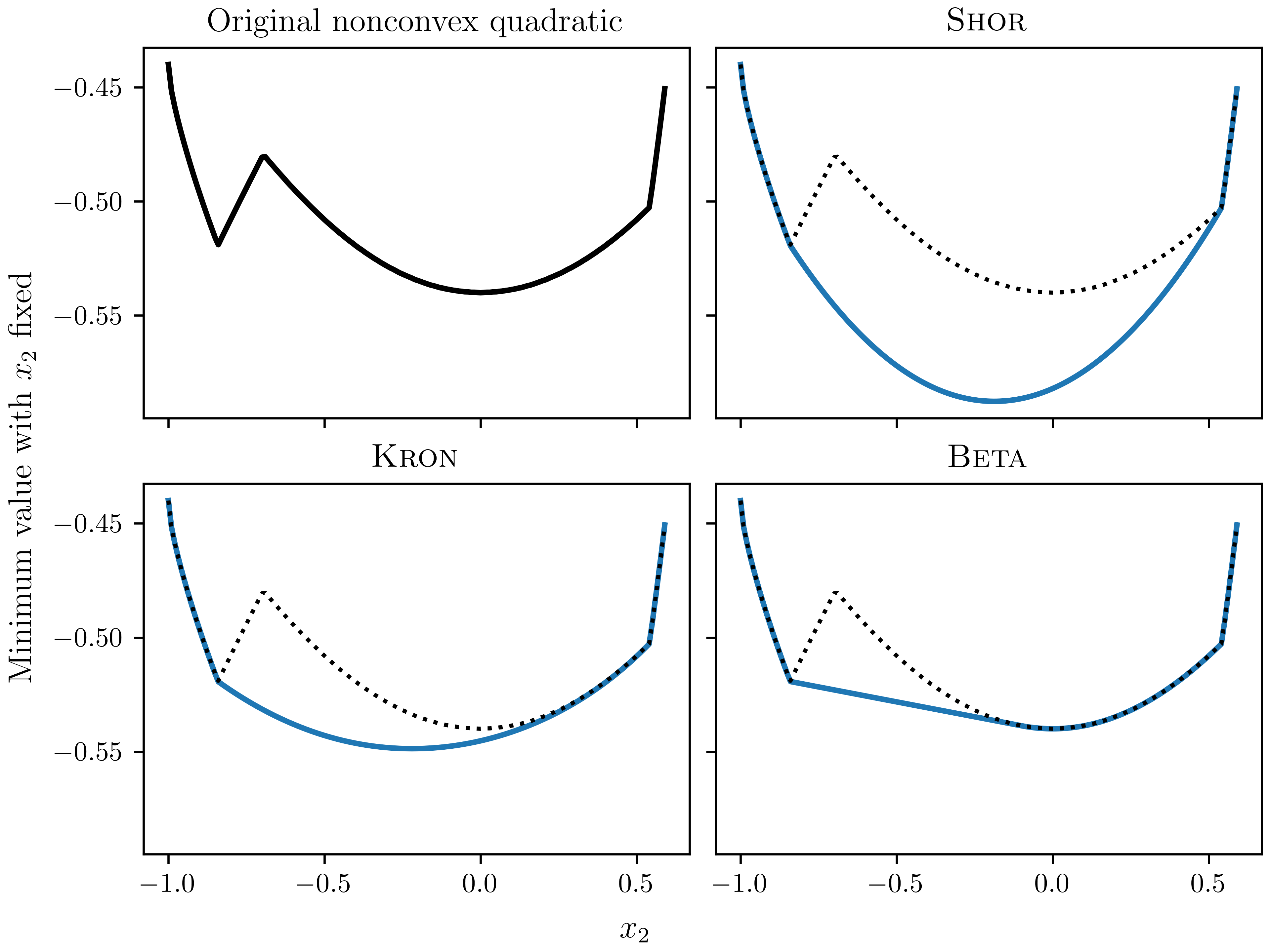}
    \caption{Illustrations of Example \ref{exa:1}}
   \label{fig:example_1}
\end{figure}

\subsection{The exact case $m=2$}

For $m = 2$, the auxiliary variable $\beta$ is inserted between $x^T
x$ and both linear functions $\rho_i^2 - c_i^T c_i + 2 c_i^T x$ for
$i=1,2$. When $x^T x$ is strictly less than both linear functions,
there are multiple values of $\beta$, which are feasible for the same
value of $x$. To remove this ambiguity, we can actually force $\beta
= \min_{i=1,2} \{ \rho_i^2 - c_i^T c_i + 2 c_i^T x \}$ by adding the
quadratic complementarity equation
\[
    (\rho_1^2 - c_1^T c_1 + 2 c_1^T x - \beta)(\rho_2^2 - c_2^T c_2 + 2
    c_2^T x - \beta) = (\ell_1^T w)(\ell_2^T w) = 0
\]
to the definition of $\F$. Formally, we define
\begin{align*}
    \F^0 
    &:= \left\{
    w \in \F :
    (\ell_1^T w) (\ell_2^T w) = 0
    \right\}, \\
    \G^0 &:= \cone \{ W = ww^T : w \in \F^0 \}, \\
    \Rc^0 &:= \{ W \in \Rc : \ell_1^T W \ell_2 = 0 \}.
\end{align*}
Note that optimizing $x^T Q x + 2 q^T x$ over $\F^0 \cap \{ {\color{black} w} : \alpha
= 1 \}$ is equivalent to (\ref{equ:origqp}) because $\beta$ is
artificial. We have the following exactness theorem, whose proof is delayed until Section \ref{sec:proofs}:

\begin{theorem} \label{the:main}
Regarding the ball-constrained set (\ref{equ:groundset}), for $m=2$, it
holds that $\G^0 = \Rc^0$.
\end{theorem}

\noindent This ensures that, for $m=2$, (\ref{equ:origqp}) is equivalent to
\[
    \min \left\{ 
        \widehat Q
        \bullet W
        :
        W \in \Rc, \ W_{11} = 1, \ \ell_1^T W \ell_2 = 0
    \right\} = 
    \min \left\{ 
        \widehat Q
        \bullet W
        :
        W \in \Rc^0, \ W_{11} = 1
    \right\},
\]
that is, equivalent to \Us~with the strengthened constraint $\ell_1^T
W \ell_2 = 0$.

To verify Theorem \ref{the:main} empirically, we generated 1,000
feasible instances of (\ref{equ:origqp}) for each of the dimension
pairs $(n,m) \in \{ (2,2), (4,2), ({\color{black} 6},2) \}$ following
the method of \cite{Burer.Anstreicher.2013}; see section 5.3 of that
paper and the discussion therein. The idea is to first generate and
optimize globally an instance of (\ref{equ:origqp}) with $m=1$ and
then to add a second ball constraint, which cuts off the optimal
solution just calculated. The resulting problem is likely to possess
multiple good candidates for optimal solutions, thus ostensibly
making it a more challenging instance of (\ref{equ:origqp}). We call
these the {\em Martinez instances\/} following the terminology in
\cite{Burer.Anstreicher.2013}.

{\color{black} Among the total 3,000 instances generated, we
specifically excluded ones that were already solved by \Shor. This
was done in order to eliminate the ``easiest'' instances, i.e., the
ones that can already be handled by the simplest relaxation. In fact,
for random instances such as the ones described, the percentage of
instances already solved by Shor increases empirically as $n$
grows. Hence, limiting the experiments to the (relatively small) subset
of instances that are not solved by Shor allows a more focused
comparison of the performance of \Kron~and \Us.}

\begin{table}[t]
\centering
\begin{tabular}{rr|r|rr|rrr}
$n$ & $m$ & \# Instances & \multicolumn{2}{c|}{\# Solved} & \multicolumn{3}{c}{Total Time (s)} \\
         &  & & \Kron & \Us & \Shor & \Kron & \Us \\ \hline
   2 &  2 & 1,000 & 955 & 1,000 & 0.4 & 2.0 & 0.8 \\ 
   4 &  2 & 1,000 & 884 & 1,000 & 0.7 & 13.5 & 1.0 \\ 
   6 &  2 & 1,000 & 882 & 1,000 & 0.9 & 181.4 & 1.6 \\ 
\end{tabular}
\caption{Number of instances of (\ref{equ:origqp}) solved by \Kron~and
\Us~over 3,000 randomly generated Martinez instances, which were not
already solved by \Shor. Instances are grouped by dimensions $(n,m)$.
Also shown are the total time (in seconds) to optimize all instances
for each method in each grouping. Times for \Shor~are also shown for
reference.}
\label{tab:martinez_performance}
\end{table}

Table \ref{tab:martinez_performance} shows the results of our
experiments. We see that, while \Kron~solves many instances, \Us~solves
all instances, as ensured by Theorem \ref{the:main}, in much less
time. In fact, \Us~takes just a bit more time than \Shor.

Eltved and Burer \cite{Eltved.Burer.2023} also optimized instances of
(\ref{equ:origqp}) with $m=2$, and they reported that their method,
which is at least as strong as \Kron, was unable to solve 267 randomly
generated instances with $n$ ranging from 2 to 10; see table 6 therein.
We ran \Us~on these instances and globally solved {\color{black} all} of
them.\footnote{These instances are available at \url{https://github.com/A-Eltved/strengthened_sdr}.}

\subsection{The general case $m \ge 3$} \label{sec:balls}

We now investigate the use of \Kron~and \Us~to approximate instances
of (\ref{equ:origqp}) with $m \ge 3$. To this end, we generated random
instances of the following problem, which we call {\em max-norm\/} and
which has recently been studied in \cite{Costandin.2023}:

\begin{quote}
In $\R^n$, given a point $p$ and $m$ balls containing the origin, find a
point in the intersection of the balls with maximum distance to $p$.
\end{quote}

\noindent This is an instance of (\ref{equ:origqp}) with $Q = -I$ and $q
= p$. We have chosen this class because it seems to be an interesting
geometric problem, which is also NP-hard \cite{Costandin.2023}.

Specifically, the first ball is taken to be the unit ball, and the
remaining $m-1$ balls are generated with random centers $c_i$ in the
unit ball. Then the random radii $\rho_i$ for $i=2,\ldots,m$ are
generated each in $\|c_i\| + U(0,1.5)$, where $U(0,1.5)$ is continuous
uniform between 0 and 1.5. In particular, this guarantees that $x=0$
is in the resulting feasible set. Finally, $p$ is generated
uniformly in the ball of radius 4 centered at the origin, and as
such could be either inside or outside the feasible set. As in the
previous subsection, we exclude instances that are solved by \Shor~in
order to remove the easiest instances and focus on the performance of
\Us~relative to \Kron.

We also examined the so-called {\em gap closure\/} for each relaxation.
For a given instance, let $s^*$ be the optimal value of \Shor, and
let $v$ be the best (i.e., minimum) feasible value from the solutions
embedded within the three relaxations \Shor, \Kron, and \Us~at
optimality. Then
\begin{align*}
    \text{gap closure for \Kron}
    &:= \left( \frac{\text{[optimal value for \Kron]} - s^*}{v - s^*} \right) \times 100\% \\
    \text{gap closure for \Us}
    &:= \left( \frac{\text{[optimal value for \Us]} - s^*}{v - s^*} \right) \times 100\%
\end{align*}
A gap closure of 0\% means that the relaxation did not improve upon
\Shor, and a gap closure of 100\% means that the relaxation was exact.

Tables \ref{tab:maxnorm_performance} and \ref{tab:maxnorm_gapclosed}
show the results of \Kron~and \Us~on 1,000 max-norm instances for
{\color{black} three} pairs of dimensions $(n,m)$. These results show
clearly that \Us~significantly outperforms \Kron~in terms of both number
of instances solved and the average gap closed. {\color{black} In terms
of timings, we see that \Us~requires significantly less time.} Furthermore,
no instances were solved by \Kron~and simultaneously
unsolved by \Us. This suggests empirically that \Us~is always at least
as strong as \Kron, although we do not have a formal proof establishing
this.

\begin{table}[t]
\centering
\begin{tabular}{rr|r|rr|rrr}
$n$ & $m$ & \# Instances & \multicolumn{2}{c|}{\# Solved} & \multicolumn{3}{c}{Total Time (s)} \\ 
         &  & & \Kron & \Us & \Shor & \Kron & \Us \\ \hline
   2 &  5 & 1,000 & 208 & 977 & 0.3 & 14.7 & 1.1 \\ 
   2 &  9 & 1,000 & 412 & 973 & 0.4 & 153.1 & 1.9 \\ 
   4 &  9 & 1,000 &  2 & 908 & 0.8 & 2435.6 & 3.9 \\ 
\end{tabular}
\caption{Number of max-norm instances solved by \Kron~and \Us~on 3,000
instances unsolved by \Shor. Instances are grouped by dimensions
$(n,m)$. Also included are the total time (in seconds) to optimize all
instances for each method in each grouping.}
\label{tab:maxnorm_performance}
\end{table}

\begin{table}[t]
\centering
\begin{tabular}{cc|r|rr}
    \multicolumn{2}{c|}{Solution Status} & \# Instances & \multicolumn{2}{c}{Avg Gap Closed} \\ 
    \Kron & \Us & & {\sc Kron} & \Us \\ 
  \hline
  unsolved & unsolved & 142 & 6\% & 29\% \\ 
  unsolved & solved & 2236 & 30\% & 100\% \\ 
  solved & unsolved & 0 & - & - \\
  solved & solved & 622 & 100\% & 100\% 
\end{tabular}
\caption{Average gap closures for Table \ref{tab:maxnorm_performance}.
Instances are grouped by solution status for \Kron~and \Us.}
\label{tab:maxnorm_gapclosed}
\end{table}

As a final experiment, to judge the time required to optimize \Us~as a
function of $n$ and $m$, we computed the mean time to optimize a single
random instance for 36 different combinations of $(n,m)$ ranging from
$(2,2)$ to $(64,64)$. For each mean calculation, the sample size was
100. From Table \ref{tab:maxnorm_larger}, we see that our relaxation
can be optimized in tens of seconds for modest values of $n$ and $m$.

\begin{table}[ht]
\centering
\begin{tabular}{r|rrrrrr}
    & \multicolumn{6}{c}{$m$} \\ 
$n$ & 2 & 4 & 8 & 16 & 32 & 64 \\ 
  \hline
   2 & 0.0 & 0.0 & 0.0 & 0.0 & 0.0 & 0.5 \\ 
   4 & 0.0 & 0.0 & 0.0 & 0.0 & 0.1 & 1.0 \\ 
   8 & 0.0 & 0.0 & 0.0 & 0.0 & 0.2 & 2.4 \\ 
  16 & 0.0 & 0.0 & 0.0 & 0.1 & 0.5 & 4.2 \\ 
  32 & 0.0 & 0.0 & 0.1 & 0.3 & 1.3 & 8.3 \\ 
  64 & 0.1 & 0.2 & 0.4 & 1.0 & 3.9 & 40.6
\end{tabular}
\caption{Average time (in seconds) to optimize the \Us~relaxation of 100
random max-norm instances for various combinations of $(n,m)$.}
\label{tab:maxnorm_larger}
\end{table}

\section{Extensions} \label{sec:extensions}

In this section, we provide two extensions, which use the $\beta$ idea
described in Section \ref{sec:newrelax}.

\subsection{Linear upper bounds on the norm}

Consider nonconvex quadratic programming over the set
\begin{equation} \label{equ:groundset_}
    \left\{ x \in \R^n :
        \begin{array}{l}
            \|x\| \le 1 \\
            \|x\| \le g + h^T x 
        \end{array}
    \right\},
\end{equation}
where $g$ is a scalar and $h$ is a vector. We assume
(\ref{equ:groundset_}) is nonempty, but otherwise no
particular assumptions are made on $g$ and $h$. Kelly et
al.~\cite{Kelly.et.al.2022} studied this case using a disjunctive SDP
approach; see also the discussion in Section \ref{sec:existing}.

As before, we introduce the auxiliary variable $\beta \in \R$ into
(\ref{equ:groundset_}), homogenize, enforce complementarity, and
consider the cone
\[
    \F ^0
    = \left\{
    w \in \R^{n+2} :
    P^T w \in \SOC^{n+1}, \
    \ell_i^T w \ge 0 \ \forall \ i = 1,2, \
    (\ell_1^T w)(\ell_2^T w) = 0
    \right\},
\]
where
\[
    P :=
    \begin{pmatrix}
        0 & 0 \\
        0 & I_n \\
        1 & 0
    \end{pmatrix}
    \in \R^{(n + 2) \times (n + 1)},
    \quad
    \ell_1 := \begin{pmatrix} 1 \\ 0 \\ -1 \end{pmatrix} \in \R^{n+2},
    \quad
    \ell_2 := \begin{pmatrix} g \\ h \\ -1 \end{pmatrix} \in \R^{n+2}.
\]
Then Section \ref{sec:background} provides the following relaxation of
$\G^0 := \cone \{ ww^T \in \Sbb^{n+1} : w \in \F^0 \}$:
\begin{align*}
    \Rc^0
    &:=
    \left\{
        W \in \PSD^{n+2}:
        \begin{array}{l}
            J_{n+1} \bullet P^T W P \ge 0 \\
            \ell_1^T W \ell_2 = 0 \\
            P^T W\ell_i \in \SOC^{n+1} \ \ \forall \ i=1,2 
        \end{array}
    \right\}.
\end{align*}
As with Theorem \ref{the:main}, we can prove that $\Rc^0$ captures the
convex hull exactly.

\begin{theorem} \label{the:extension}
Regarding the feasible set (\ref{equ:groundset_}), it holds that $\G^0 =
\Rc^0$.
\end{theorem}

\begin{proof}
    The same proof for Theorem \ref{the:main} handles this case as well.
\end{proof}

\subsection{Two Trust Region Subproblem}

As another extension, we consider the problem of minimizing a nonconvex
quadratic over the intersection of two full-dimensional ellipsoids, i.e., the TTRS
mentioned in the Introduction:
\begin{equation} \label{equ:ttrs} 
    \min_{x \in \R^n}
    \left\{
        x^T Q x + 2 \, q^T x :
        \begin{array}{l}
            \|H_i x - c_i \| \le \rho_i \quad \forall \ i=1,2
        \end{array}
    \right\},
\end{equation}
where $H_i$ are full-rank matrices and $c_i,\rho_i$ are as in
(\ref{equ:origqp}). By a change of variables, we may assume without
loss of generality that $H_1 = I_n$, $c_1 = 0$, and $\rho_1 = 1$, i.e.,
that the first ellipsoid is the unit ball. As with the ball case, the
relaxation $\Shor \cap \Kron$ is applicable; see, in
particular, the study by Anstreicher \cite{Anstreicher.2017}.

By another change of variables, we can further assume that $H_2$ is
diagonal without loss of generality, in which case the feasible set can
be rewritten as
\[
    \left\{ x \ : \ x^T x \le 1, \ \ 
    \sum_{j=1}^n [H_2]_j^2 x_j^2 \le \rho_2^2 - c_2^T c_2 + 2 c_2^T H_2 x.
    \right\}
\]
We next introduce $n$ new artificial variables $\beta_j$ satisfying
$x_j^2 \le \beta_j$ and then homogenize and enforce complementarity:
\[
    \F := \left\{
        w  = \begin{pmatrix} \alpha \\ x \\ \beta \end{pmatrix} 
        \in \mathbb{R}^{2n + 1} 
        \ : \
        \begin{array}{l}
        P_j^T w \in \SOC^3 \quad \forall \ j=1,\ldots,n \\
        \ \ell_i^T w \ge 0 \quad \forall \ i =1,2 \\
        (\ell_1^T w)(\ell_2^T w) = 0
        \end{array}
    \right\},
\]
where $e_j \in \R^n$ is the $j$-th standard unit vector,
\[
    P_j :=
    \begin{pmatrix}
        1/2 & 0 & 0 \\
        0 & 0 & e_j \\
        0 & e_j & 0
    \end{pmatrix}
    T_3
    \in \R^{(2n+1) \times 3} \ \ \forall \ j =1,\ldots,n,
\]
and
\[
    \ell_i := 
    \begin{pmatrix} \rho_i^2 - c_i^T c_i \\ 2 H_i c_i \\ -\diag(H_i^2) \end{pmatrix} \in \R^{2n+1}
    \ \ \forall \ i=1,2.
\]
Note that, this formulation has $2n+1$ variables, $n$
SOC constraints of size 3, two linear constraints, and one
complementarity constraint. Following the development in Section
\ref{sec:background}, our relaxation is then
\begin{align*}
    \Rc^0
    &:=
    \left\{
        W \in \PSD^{2n + 1}:
        \begin{array}{l}
            J_3 \bullet P_j^T W P_j \ge 0 \ \ \forall \ j = 1,\ldots,n \\
            \ell_1^T W \ell_2 = 0 \\
            P_j^T W\ell_i \in \SOC^3 \ \ \forall \ i=1,2, \ j=1,\ldots,n \\
            (\Arr_3 \boxtimes \Arr_3)(P_k^T W P_j) \in \PSD^{9} \ \
            \forall \ 1 \le j < k \le n
        \end{array}
    \right\}.
\end{align*}

The paper \cite{Burer.Anstreicher.2013} studied a relaxation of
(\ref{equ:ttrs}), which was based on rewriting the unit ball as the
intersection of a semi-infinite number of half-spaces:
\[
    \left\{ x : \begin{array}{l} \|x\| \le 1 \\ \|H_2 - c_2\| \le \rho_2 \end{array} \right\} 
    = \left\{ x : \begin{array}{l} u^T x \le 1 \ \forall \ \|u\| = 1 \\ \|H_2 - c_2\| \le \rho_2 \end{array}\right\}.
\]
In doing so, the authors constructed a relaxation, which was a
semi-infinite analog of $\Shor \cap \RLT \cap \SOCRLT$---as introduced
in Section \ref{sec:background}---which could nevertheless be optimized
in polynomial-time in this special case. As a byproduct of this
research, the authors introduced a collection of 212 instances of
(\ref{equ:ttrs}) with sizes $n \in \{5,10,20\}$, which were not solved
by $\Shor \cap \RLT \cap \SOCRLT$.\footnote{These instances are available
at \url{https://github.com/sburer/soctrust}.}

These 212 instances have been revisited by Anstreicher
\cite{Anstreicher.2017}, who showed that $\Shor \cap \Kron$ solves 77.
Another approach by Yang and Burer \cite{Burer.Yang.2015} solves 29
instances. Finally, a recent cut-generation approach by Consolini and
Locatelli \cite{Consolini.Locatelli.2021} solves 211, i.e., all but one
of the 212 instances. We applied our relaxation $\Rc^0$ to the same
instances and solved all 212. The largest instances (size $n=20$) were
solved in 0.374 seconds on average.

\section{Proof of Theorem \ref{the:main}} \label{sec:proofs}

In this section, we prove Theorem \ref{the:main} from Section
\ref{sec:newrelax}, which states that
\[
    \G^0 := 
    \cone \left\{
        ww^T \in \Sbb^{n+2} :
        \begin{array}{l}
            P^T w \in \SOC^{n+2} \\
            \ell_i^T w \ge 0 \ \ \forall \ i=1,2 \\
            (\ell_1^T w)(\ell_2^T w) = 0
        \end{array}
    \right\}
\]
equals 
\[
    \Rc^0 :=
    \left\{
        W \in \PSD^{n+2}:
        \begin{array}{l}
            J_{n+2} \bullet P^T W P \ge 0 \\
            P^T W\ell_i \in \SOC^{n+2} \ \ \forall \ i=1,2 \\
            \ell_1^T W \ell_2 = 0
        \end{array}
    \right\}.
\]
A closely related result is {\color{black} T}heorem 2.7 in \cite{Ye.Zhang.2003a}, but
our proof follows a technique from \cite{Burer.2015}.

We first prove some lemmas. To simplify notation, let $J := J_{n+2}$
and $\SOC := \SOC^{n+2}$ throughout the rest of this section. Our first
lemma states some straightforward properties of $J$.

\begin{lemma} \label{lem:propJ}
Regarding $J$, it holds that:
\begin{enumerate}
    \item[(i)] If $v \in \SOC$, then $\myJ v \in \SOC$, and hence
        $w^T \myJ  v \ge 0$ for all $w \in \SOC$.
    \item[(ii)] If $v \in \bd(\SOC)$, then $v^T \myJ  v = 0$.
\end{enumerate}
\end{lemma}

\noindent The next two lemmas are {\color{black} analogous} to the
exactness results of Shor~and $\Shor \cap \SOCRLT$ discussed in Section
\ref{sec:existing}.

\begin{lemma} \label{lem:trs}
$\cone \{ ww^T : P^T w \in \SOC \} = \{ W \in \PSD^{n+2} : \myJ
\bullet P^T W P \ge 0 \}.$
\end{lemma}

\begin{proof}
This can be proven, for example, by showing that all extreme rays of
the right-hand-side set have rank-1, which can in turn be shown via
Pataki's bound on the rank of extreme matrices of SDP-representable sets
\cite{Pataki.1998}.
\end{proof}

\begin{lemma} \label{lem:basecase}
Fix $i=1$ or $i=2$. It holds that
\[
    \cone \left\{ ww^T :
        \begin{array}{l}
            P^T w \in \SOC \\
            \ell_i^T w \ge 0
        \end{array}
    \right\} =
    \left\{
        W \in \PSD^{n+2} :
        \begin{array}{l}
            \myJ  \bullet P^T W P \ge 0 \\
            P^T W \ell_i \in \SOC
        \end{array}
    \right\}.
\]
\end{lemma}

\begin{proof} In the statement of the proposition, let $\G[i]$ be the
left-hand-side set, and let $\Rc[i]$ be the right-hand side. Clearly
$\G[i] \subseteq \Rc[i]$. We show the reverse inclusion by proving
$\Rc[i]$ has rank-1 extreme rays. Indeed, let $W \ne 0$ be an arbitrary
extreme ray in $\Rc[i]$. We will show $\rank(W) = 1$ by considering
three cases{\color{black} .}

First consider when $P^T W \ell_i \in \myint(\SOC)$, in which case
$W$ must also be extreme in $\{ W \in \PSD^{n+2} : \myJ \bullet P^T W P
\ge 0 \}$. Then $\rank(W) = 1$ by Lemma \ref{lem:trs}.

Next consider when $P^T W \ell_i \in \bd( \SOC)$ with $W
\ell_i = 0$. Using Lemma \ref{lem:trs}, we write $W = \sum_k w^k
(w^k)^T$ for $w^k$ satisfying $P^T w^k \in \SOC$. We then have
\[
    W \ell_i = 0
    \quad \Rightarrow \quad
    \ell_i^T W \ell_i = 0
    \quad \Rightarrow \quad
    \sum_k (\ell_i^T w^k)^2 = 0
    \quad \Rightarrow \quad
    \ell_i^T w^k = 0 \ \ \forall \ k,
\]
which proves $W \in \G[i]$ and hence $W$ is extreme in $\G[i]$. Thus,
$\rank(W) = 1$, as desired.

Finally, consider when $P^T W \ell_i \in \bd( \SOC)$ with
$W \ell_i \ne 0$. Define $v := W \ell_i \ne 0$ so that $P^T v \in
\bd(\SOC)$. In addition, $W \in \PSD^{n+2}$ implies $\ell_i^T v =
\ell_i^T W \ell_i \ge 0$. Hence, we conclude that $vv^T$ is a nonzero
member of $\G[i] \subseteq \Rc[i]$. Next, for small $\epsilon > 0$,
define $W_\epsilon := W - \epsilon \, vv^T$; we claim $W_\epsilon \in
\Rc[i]$. Indeed, $W_\epsilon \in \PSD^{n+2}$ because it is a rank-1
perturbation of $W \in \PSD^{n+2}$ with $v \in \Range(W)$ \cite[{\color{black} L}emma
2]{Burer.etal.2009}. Moreover, $P^T v \in \bd(\SOC)$ implies by
Lemma \ref{lem:propJ}(ii) that
\[
    \myJ \bullet P^T W_\epsilon P = \myJ \bullet P^T W P - \epsilon (P^T
    v)^T \myJ (P^T v) = \myJ \bullet P^T W P - \epsilon \cdot 0 \ge 0.
\]
We also have
\[
    P^T W_\epsilon \ell_i = P^T W \ell_i - \epsilon (\ell_i^T v) P^T v =
    (1 - \epsilon \cdot \ell_i^T v) P^T v \in \bd(\SOC).
\]
Thus, when $\epsilon > 0$ is small, $W_\epsilon \in \Rc[i]$ as claimed.
Then the equation $W = W_\epsilon + \epsilon \, vv^T$ and the fact that
$W$ is extreme in $\Rc[i]$ imply $W$ must be a positive multiple of
$vv^T$, i.e., $\rank(W)= 1$, as desired.
\end{proof}

\noindent Our next lemma is a technical result about extreme rays in the
intersection of two closed convex cones.

\begin{lemma} \label{lem:ext2}
Let $\P$ be a closed convex cone, and let $\Q$ be a half-space
containing the origin. Every extreme ray of $\P \cap \Q$ is either an
extreme ray of $\P$ or can be expressed as the sum of two extreme rays
of $\P$.
\end{lemma}

\begin{proof}
See \cite[Lemma 5]{Burer.2015}.
\end{proof}

We are now ready to prove Theorem \ref{the:main}.

\begin{proof}
Since $\G^0 \subseteq \Rc^0$ by construction, we show the reverse
inclusion by proving that every extreme ray $W$ of $\Rc^0$ has rank
1 and hence is an element of $\G^0$. We define $v_i := W \ell_i$ for
$i=1,2$. Note that $v_1 v_1^T \in \G^0$ because $\ell_1^T v_1 = \ell_1^T
W \ell_1 \ge 0$, $\ell_2^T v_1 = \ell_2^T W \ell_1 = 0$, and $P^T v_1 =
P^T W \ell_1 \in \SOC$. A similar argument shows $v_2 v_2^T \in
\G^0$.

We first consider the case when $v_1 = 0$. Applying Lemma
\ref{lem:basecase} for the case $i=2$, we express $W$ as
\[
    W = \sum_k w^k (w^k)^T, \quad P^T w^k \in \SOC \text{ and }
    \ell_2^T w^k \ge 0 \quad \forall \ k.
\]
As in the proof of Lemma \ref{lem:basecase}, $v_1 = 0$ then implies
$\ell_1^T w^k = 0$ for all $k$. So $W \in \G^0$, and because $W$ is
extreme in $\Rc^0 \supseteq \G^0$, it must have rank 1. A similar
argument shows $\rank(W) = 1$ for the case $v_2 = 0$.

So we assume from this point forward that $v_i := W \ell_i \ne 0$
for both $i=1,2$. Since $W \in \PSD^{n+2}$ ensures $W \ell_i = 0
\Leftrightarrow \ell_i^T W \ell_i = 0$, we have $\ell_i^T v_i = \ell_i^T
W \ell_i > 0$ for both $i$.

The second case we consider assumes $\myJ \bullet P^T W P = 0$ and $P^T
v_i = P^T W \ell_i \in \myint(\SOC)$ for both $i=1,2$. Then $W$ is
extreme for the equality-constrained cone $\{ W \in \PSD^{n+2} : \myJ
\bullet P^T W P = 0, \ell_1^T W \ell_2 = 0 \}$, which in turn implies
that $W$ is extreme for the inequality-constrained cone $\P \cap \Q$,
where
\[
    \P := \left\{ W \in \PSD^{n+2} : \myJ  \bullet P^T W P \ge 0
    \right\}, \ \ \ \
    \Q:= \{ W \in \Sbb^{n+2} : \ell_1^T W \ell_2 \ge 0 \},
\]
Applying Lemma \ref{lem:trs} with $\P$ and Lemma~\ref{lem:ext2} with $\P
\cap \Q$, we conclude that $\rank(W) \le 2$. If its rank equals 1, we
are done. So assume $\rank({\color{black} W}) = 2$. We derive a contradiction to the
assumption that $W$ is extreme in $\Rc^0$. Consider the equation
\[
    U := \begin{pmatrix} \ell_1^T \\ \ell_2^T \\ I \end{pmatrix} W \begin{pmatrix} \ell_1 & \ell_2 & I \end{pmatrix}
    =
    \begin{pmatrix} \ell_1^T W\ell_1 & \ell_1^T W\ell_2 & \ell_1^T W \\
        \ell_2^T W\ell_1 & \ell_2 W\ell_2 & \ell_2^T W \\
        W\ell_1          & W\ell_2          & W 
    \end{pmatrix} =
    \begin{pmatrix}
        \ell_1^T v_1 & 0 & v_1^T \\
        0   & \ell_2^T v_2 & v_2^T \\
        v_1 & v_2         & W
    \end{pmatrix},
\]
and recall that $\ell_1^T v_1 > 0$ and $\ell_2^T v_2 > 0$. It holds
that $U$ is PSD with $\rank(U) \le \rank(W) = 2$, {\color{black} and hence
the $\Range(W)$ is spanned by $\{v_1, v_2\}$ because the first
two columns of $U$ are clearly linearly independent. Then, because $W \succeq
0$ with $\rank(W) = 2$, it must hold that $$W = \lambda_1 v_1 v_1^T +
\lambda_2 v_2 v_2^T,$$ for some positive multipliers $\lambda_i > 0$.}
However, this contradicts the assumption that $W$ is extreme in $\Rc^0$
due to the fact that both $v_1 v_1^T$ and $v_2 v_2^T$ are elements of
$\G^0 \subseteq \Rc^0$.

For our third and final case, we assume $\myJ  \bullet P^T W P > 0$, $P^T
v_1 \in \bd(\SOC)$, or $P^T v_2 \in
\bd(\SOC)$. Let us consider two perturbations of $W$:
\[
    W_{\epsilon_i} := W - \epsilon_i \, v_i v_i^T \quad \forall \ i=1,2
\]
for two parameters $\epsilon_i > 0$. We claim that $W_{\epsilon_i} \in
\Rc^0$ for at least one $i$, in which case $W$ must be rank-1 as argued
in the proof of Lemma \ref{lem:basecase}.

Using $W \in \Rc^0$, we first argue that each $W_{\epsilon_i}$ satisfies
all constraints of $\Rc^0$, except possibly $\myJ  \bullet P^T W_{\epsilon_i}
P \ge 0$. Fix $i = 1$; the proof for $i=2$ is similar. We know
$W_{\epsilon_1} \in \PSD^{n+2}$ since $v_1 \in \Range(W)$ \cite[Lemma
2]{Burer.etal.2009}. We also have
\[
    P^T W_{\epsilon_1} \ell_1 =
    P^T v_1 - \epsilon_1 P^T v_1 (v_1^T \ell_1) =
    (1 - \epsilon_1 (v_1^T \ell_1)) P^T v_1
    \in \SOC
\]
for small $\epsilon_1 > 0$. Furthermore, noting that $v_1^T \ell_2 =
\ell_1^T W \ell_2 = 0$, we see
\[
    P^T W_{\epsilon_1} \ell_2 = 
    P^T v_2 - \epsilon_1 P^T v_1 (\ell_1^T W \ell_2) =
    P^T v_2 - \epsilon_1 P^T v_1 \cdot 0 = P^T v_2 
    \in \SOC.
\]
Finally,
\[
    \ell_1^T W_{\epsilon_1} \ell_2
    = \ell_1^T W \ell_2 - \epsilon_1 (\ell_1^T v_1)(v_1^T \ell_2) 
    = \ell_1^T W \ell_2 - \epsilon_1 (\ell_1^T v_1)(\ell_1^T W \ell_2) 
    = 0 - \epsilon_1 \cdot \ell_1^T v_1 \cdot 0 = 0,
\]
as desired.

We now claim that at least one $W_{\epsilon_i}$ satisfies the remaining
constraint
\[
    \myJ  \bullet P^T W_{\epsilon_i} P = \myJ  \bullet P^T W P
    - \epsilon_i \myJ  \bullet (P^T v_i) (P^T v_i)^T \ge 0
\]
for small $\epsilon_i > 0$, thus completing the proof as discussed
above. If $\myJ \bullet P^T W P > 0$, then both $W_{\epsilon_i}$
satisfy the inequality. If $P^T v_1 \in \bd(\SOC)$, then
$W_{\epsilon_1}$ satisfies the inequality because $\myJ \bullet (P^T
v_i)(P^T v_i)^T = 0$ by Lemma \ref{lem:propJ}(ii). Similarly if $P^T v_2
\in \bd(\SOC)$, then $W_{\epsilon_2}$ satisfies the inequality.
\end{proof}

\section{Conclusions}

In this paper, we have constructed strong relaxations for
(\ref{equ:origqp}) by transforming its feasible set, lifting to one
higher dimension, and employing standard relaxation techniques from
the literature. Our computational results demonstrate the strength of
our relaxation, and the time required for solving our relaxation is
small compared to those in the literature. We have also shown that our
relaxation is exact for the case of two balls.

There are many open questions related to this research. For example, is
it possible to prove that our relaxation \Us~is always at least as tight
as \Kron, as supported by the computational evidence? Furthermore,
is \Us~provably as strong as the method of \cite{Eltved.Burer.2023}?
Is it possible to extend Theorem \ref{the:main} to the case of more
constraints? Finally, even when our relaxation is not exact in practice
for $m \ge 3$, could it be used within an effective global optimization
algorithm of (\ref{equ:origqp})?

\section*{Statements and Declarations}

The author declares that he has no competing interests.

\end{onehalfspace}

\section*{Acknowledgments}

The author expresses his sincere thanks to Kurt Anstreicher for an
important observation, which ultimately led to the establishment of
Theorem \ref{the:main}. Thanks are also extended to the anonymous
reviewers and editors, whose suggestions have improved this paper
immensely.

\bibliographystyle{abbrv}
\bibliography{paper}

\end{document}